\documentclass[a4paper,11pt]{article}%
\usepackage{amssymb}
\usepackage{amsmath}
\usepackage{theorem}
\usepackage{amsfonts}
\usepackage{graphicx}%

\newtheorem{theorem}{Theorem}
\newtheorem{corollary}[theorem]{Corollary}
\newtheorem{definition}[theorem]{Definition}

\newtheorem{lemma}[theorem]{Lemma}
\newtheorem{problem}[theorem]{Problem}
\newtheorem{proposition}[theorem]{Proposition}
\newtheorem{remark}[theorem]{Remark}
\newenvironment{proof}[1][Proof]{\noindent\textbf{#1.} }{\ \rule{0.5em}{0.5em}}
\begin{document}

\author{
	Tomer Kotek\\
	Department of Computer Science\\
	Technion - Israel Institute of Technology\\
	Haifa, Israel
	\and 
	James Preen\\
	Mathematics\\
	Cape Breton University\\
	Sydney, Canada
	\and Peter Tittmann\\
	Faculty Mathematics, Sciences, and Computer Science\\
	Hochschule Mittweida - University of Applied Sciences\\
	Mittweida, Germany
	}
	
\title{Subset-Sum Representations of Domination Polynomials}
\maketitle

\begin{abstract}
The domination polynomial $D(G,x)$ is the ordinary generating function for the
dominating sets of an undirected graph $G=(V,E)$ with respect to their
cardinality. We consider in this paper representations of $D(G,x)$ as a sum
over subsets of the edge and vertex set of $G$. One of our main results is a
representation of $D(G,x)$ as a sum ranging over spanning bipartite subgraphs
of $G$.

Let $d(G)$ be the number of dominating sets of $G$. We call a graph $G$
\emph{conformal} if all of its components are of even order. Let
$\mathrm{Con}(G)$ be the set of all vertex-induced conformal subgraphs of $G$
and let $k(G)$ be the number of components of $G$. We show that%
\[
d(G)=\sum_{H\in\mathrm{Con}(G)}2^{k(H)}.
\]

\end{abstract}

\section{Introduction}

Let $G=(V,E)$ be an undirected graph. All graphs considered in this paper are
assumed to be finite and simple. The \emph{closed neighborhood} $N_{G}\left[
v\right]  $ of a vertex $v\in V$ is the set consisting of $v$ and all its
neighbor vertices in $G$. For any subset $W\subseteq V$, we denote by
$N_{G}\left[  W\right]  $ the \emph{closed neighborhood} of $W$ in $G$, that is%

\[
N_{G}\left[  W\right]  =\bigcup\limits_{v\in W}N_{G}\left[  v\right]  .
\]

If the graph is clear from the context, then we write $N\left[  v\right]  $
and $N\left[  W\right]  $ instead of $N_{G}\left[  W\right]  $ and
$N_{G}\left[  v\right]  $, respectively. A \emph{dominating set} of $G$ is a
vertex subset $W\subseteq V$ such that $N\left[  W\right]  =V$. Let
$W\subseteq V$ be a given vertex subset of the graph $G=(V,E)$. We denote by
$\partial(W)$ the set of all edges of $G$ that have exactly one of their end
vertices in $W$, that is%

\[
\partial(W)=\left\{  \left\{  u,v\right\}  \in E\mid u\in W,v\in V\setminus
W\right\}  .
\]

The edges of $\partial(W)$ link vertices of $W$ with vertices of $V\setminus
W$. Whether a given set $W$ is a dominating set of $G$ depends neither on
edges lying completely inside $W$ nor on edges that have no end vertex in $W$,
which gives the following statement.

\begin{proposition}
\label{prop_delta}Let $G=(V,E)$ be a graph, $W\subseteq V$, and $F\subseteq E
$. Then $W$ is a dominating set of $\left(  V,F\right)  $ if and only $W$ is
dominating in $\left(  V,F\cap\partial(W)\right)  $, i.e.%

\[
N_{\left(  V,F\right)  }\left[  W\right]  =V\ \Longleftrightarrow\ N_{\left(
V,F\cap\partial(W)\right)  }\left[  W\right]  =V.
\]

\end{proposition}

\begin{definition}
Let $G=(V,E)$ be an undirected graph and $d_{k}(G)$ the number of dominating
sets of cardinality $k$ in $G$ for $k=0,...,n=\left\vert V\right\vert $. The
\emph{domination polynomial of }$G$ is%

\[
D(G,x)=\sum_{k=0}^{n}d_{k}(G)x^{k}.
\]

\end{definition}

We denote by $d(G)$ the number of dominating sets of $G$. Consequently, we
find $d(G)=D(G,1)$.

The domination polynomial of a graph has been introduced by Arocha and Llano
in \cite{Arocha}. More recently it has been investigated with respect to
special graphs, zeros, and applications in network reliability, see
\cite{AAOP,AP,AP2,AAP,DT}.

The domination polynomial can also be represented as a sum over vertex subsets
of $G$,%
\[
D(G,x)=\sum_{\substack{U\subseteq V \\N\left[  U\right]  =V}}x^{\left\vert
U\right\vert }.
\]

The domination polynomial is multiplicative with respect to components, see
\cite{Arocha}. Let $G_{1},...,G_{k}$ be the components of a given graph $G$,
then%
\begin{equation}
D(G,x)=%
{\displaystyle\prod\limits_{i=1}^{k}}
D(G_{i},x). \label{product}%
\end{equation}

\section{Spanning Subgraphs}

In this section, we provide a representation of the domination polynomial as a
sum ranging over all bipartite spanning subgraphs of a graph.

\subsection{Connected Bipartite Graphs}

Alternating sums of domination polynomials of spanning subgraphs of a given
graph yield a particularly simple result in case of connected bipartite graphs.

\begin{lemma}
\label{Lemma_bipart}Let $G=(V,E)$ be a connected bipartite graph with
bipartition $V=Y\cup Z$, $Y\neq\emptyset$, $Z\neq\emptyset$. Then%
\[
\sum_{F\subseteq E}(-1)^{\left\vert F\right\vert }D\left(  \left(  V,F\right)
,x\right)  =(-1)^{\left\vert Y\right\vert }x^{\left\vert Z\right\vert
}+(-1)^{\left\vert Z\right\vert }x^{\left\vert Y\right\vert }.
\]

\end{lemma}

\begin{proof}
Let $W$ be a dominating set of $G$; then we can distinguish three cases, namely%

\begin{align*}
\text{(a)}  &  W\cap Y\neq\emptyset\text{ and }W\cap Z\neq\emptyset,\\
\text{(b)}  &  W=Y,\\
\text{(c)}  &  W=Z.
\end{align*}

We decompose the sum according to the above given cases:%

\begin{align}
\sum_{F\subseteq E}(-1)^{\left\vert F\right\vert }D\left(  \left(  V,F\right)
,x\right)   &  =\sum_{F\subseteq E}\sum_{\substack{W\subseteq V \\N_{\left(
V,F\right)  }\left[  W\right]  =V}}(-1)^{\left\vert F\right\vert
}x^{\left\vert W\right\vert }\nonumber\\
&  =\sum_{W\subseteq V}x^{\left\vert W\right\vert }\sum_{\substack{F\subseteq
E \\N_{\left(  V,F\right)  }\left[  W\right]  =V}}(-1)^{\left\vert
F\right\vert }\nonumber\\
&  =\sum_{\substack{W\subseteq V \\W\cap Y\neq\emptyset\\W\cap Z\neq\emptyset
}}x^{\left\vert W\right\vert }\sum_{\substack{F\subseteq E \\N_{\left(
V,F\right)  }\left[  W\right]  =V}}(-1)^{\left\vert F\right\vert }%
\tag{a}\label{sum1}\\
&  +x^{\left\vert Y\right\vert }\sum_{\substack{F\subseteq E \\N_{\left(
V,F\right)  }\left[  Y\right]  =V}}(-1)^{\left\vert F\right\vert }%
\tag{b}\label{sum2}\\
&  +x^{\left\vert Z\right\vert }\sum_{\substack{F\subseteq E \\N_{\left(
V,F\right)  }\left[  Z\right]  =V}}(-1)^{\left\vert F\right\vert }.
\tag{c}\label{sum3}%
\end{align}

We show that the Sum (\ref{sum1}) vanishes. According to Proposition
\ref{prop_delta}, a set $W$ is dominating in $(V,F)$ if and only if $W$ is a
dominating set of $(V,F\cap\partial(W))$. The evaluation of the Sum
(\ref{sum1}) yields%

\begin{align*}
\sum_{\substack{W\subseteq V \\W\cap Y\neq\emptyset\\W\cap Z\neq\emptyset
}}x^{\left\vert W\right\vert }\sum_{\substack{F\subseteq E \\N_{\left(
V,F\right)  }\left[  W\right]  =V}}(-1)^{\left\vert F\right\vert }  &
=\sum_{\substack{W\subseteq V \\W\cap Y\neq\emptyset\\W\cap Z\neq\emptyset
}}x^{\left\vert W\right\vert }\sum_{\substack{F_{1}\subseteq E\cap\partial(W)
\\F_{2}\subseteq E\setminus\partial(W) \\N_{\left(  V,F_{1}\right)  }\left[
W\right]  =V}}(-1)^{\left\vert F_{1}\cup F_{2}\right\vert }\\
&  =\sum_{\substack{W\subseteq V \\W\cap Y\neq\emptyset\\W\cap Z\neq\emptyset
}}x^{\left\vert W\right\vert }\sum_{\substack{F_{1}\subseteq E\cap\partial(W)
\\N_{\left(  V,F_{1}\right)  }\left[  W\right]  =V}}(-1)^{\left\vert
F_{1}\right\vert }\sum_{F_{2}\subseteq E\setminus\partial(W)}(-1)^{\left\vert
F_{2}\right\vert }.
\end{align*}

Now assume that $E\setminus\partial(W)=\emptyset$. Let $y\in Y\cap W$ and
$z\in Z\cap W$. Then there does not exist a path between $y$ and $z$ in $G$.
This contradicts the assumed connectedness of $G$; hence $E\setminus
\partial(W)\neq\emptyset$, which gives%
\[
\sum_{F_{2}\subseteq E\setminus\partial(W)}(-1)^{\left\vert F_{2}\right\vert
}=(1-1)^{\left\vert E\setminus\partial(W)\right\vert }=0.
\]
Now we turn to the Sum (\ref{sum2}),%
\[
\sum_{\substack{F\subseteq E\\N_{\left(  V,F\right)  }\left[  Y\right]
=V}}(-1)^{\left\vert F\right\vert }.
\]
An edge subset $F\subseteq E$ satisfies the property \textquotedblleft$Y$ is
dominating in $\left(  V,F\right)  $\textquotedblright\ if and only if $F$
contains at least one edge from each vertex of $Z$. We denote the vertices of
$Z$ by $v_{1},...,v_{k}$. For each $i$, $i=1,...,k$, let $E_{i}$ be the set of
edges of $G$ that are incident to $v_{i}$. We define%
\[
\mathcal{F}=\left\{  A\subseteq E\mid\forall i=1,...,k:\left\vert E_{i}\cap
A\right\vert \geq1\right\}  .
\]
Now the Sum (\ref{sum2}) can be expressed as follows,%
\begin{align*}
\sum_{\substack{F\subseteq E\\N_{\left(  V,F\right)  }\left[  Y\right]
=V}}(-1)^{\left\vert F\right\vert }  &  =\sum_{F\in\mathcal{F}}%
(-1)^{\left\vert F\right\vert }\\
&  =\sum_{\substack{F_{1}\cup F_{2}\cup...\cup F_{k}\in\mathcal{F}\\\forall
i=1,...,k:F_{i}\subseteq E_{i}}}(-1)^{\left\vert F_{1}\cup F_{2}\cup...\cup
F_{k}\right\vert }\\
&  =\sum_{\forall i=1,...,k:\emptyset\neq F_{i}\subseteq E_{i}}%
(-1)^{\left\vert F_{1}\right\vert +\left\vert F_{2}\right\vert +...+\left\vert
F_{k}\right\vert }\\
&  =\sum_{\substack{F_{1}\subseteq E_{1}\\F_{1}\neq\emptyset}}(-1)^{\left\vert
F_{1}\right\vert }\sum_{\substack{F_{2}\subseteq E_{2}\\F_{2}\neq\emptyset
}}(-1)^{\left\vert F_{2}\right\vert }\cdot\cdot\cdot\sum_{\substack{F_{k}%
\subseteq E_{k}\\F_{k}\neq\emptyset}}(-1)^{\left\vert F_{k}\right\vert }\\
&  =(-1)^{k}=(-1)^{\left\vert Z\right\vert },
\end{align*}
which yields%
\[
x^{\left\vert Y\right\vert }\sum_{\substack{F\subseteq E\\N_{\left(
V,F\right)  }\left[  Y\right]  =V}}(-1)^{\left\vert F\right\vert
}=(-1)^{\left\vert Z\right\vert }x^{\left\vert Y\right\vert }.
\]
In the same vein, we can prove that the sum (c) satisfies%
\[
x^{\left\vert Z\right\vert }\sum_{\substack{F\subseteq E\\N_{\left(
V,F\right)  }\left[  Z\right]  =V}}(-1)^{\left\vert F\right\vert
}=(-1)^{\left\vert Y\right\vert }x^{\left\vert Z\right\vert }%
\]
and the statement follows.
\end{proof}

\subsection{General Bipartite Graphs}

\begin{lemma}
\label{lemma_bipart2}Let $G=(V,E)$ be a bipartite graph with bipartition
$V=Y\cup Z$. Assume that $G$ consists of $k+l$ components such that the $k$
components $G_{1}=(V_{1},E_{1}),...,G_{k}=(V_{k},E_{k})$ have nonempty edge
sets and the remaining $l$ components are isomorphic to $K_{1}$. Then
\[
\sum_{F\subseteq E}(-1)^{\left\vert F\right\vert }D\left(  \left(  V,F\right)
,x\right)  =x^{l}%
{\displaystyle\prod\limits_{i=1}^{k}}
\left[  (-1)^{\left\vert Y\cap V_{i}\right\vert }x^{\left\vert Z\cap
V_{i}\right\vert }+(-1)^{\left\vert Z\cap V_{i}\right\vert }x^{\left\vert
Y\cap V_{i}\right\vert }\right]  .
\]

\end{lemma}

\begin{proof}
For the one-vertex graph $K_{1}=\left(  \left\{  v\right\}  ,\emptyset\right)
$, we obtain%
\[
\sum_{F\subseteq\emptyset}(-1)^{\left\vert F\right\vert }D\left(  \left(
\left\{  v\right\}  ,F\right)  ,x\right)  =x.
\]
By Equation (\ref{product}), we obtain%
\begin{align*}
\sum_{F\subseteq E}(-1)^{\left\vert F\right\vert }D\left(  \left(  V,F\right)
,x\right)   &  =\sum_{F\subseteq E}(-1)^{\left\vert F\right\vert }x^{l}%
{\displaystyle\prod\limits_{i=1}^{k}}
D\left(  \left(  V_{i},F\cap E_{i}\right)  ,x\right) \\
&  =x^{l}\sum_{F\subseteq E}%
{\displaystyle\prod\limits_{i=1}^{k}}
(-1)^{\left\vert F\cap E_{i}\right\vert }D\left(  \left(  V_{i},F\cap
E_{i}\right)  ,x\right) \\
&  =x^{l}%
{\displaystyle\prod\limits_{i=1}^{k}}
\sum_{F\subseteq E_{i}}(-1)^{\left\vert F\right\vert }D\left(  \left(
V_{i},F\right)  ,x\right) \\
&  =x^{l}%
{\displaystyle\prod\limits_{i=1}^{k}}
\left[  (-1)^{\left\vert Y\cap V_{i}\right\vert }x^{\left\vert Z\cap
V_{i}\right\vert }+(-1)^{\left\vert Z\cap V_{i}\right\vert }x^{\left\vert
Y\cap V_{i}\right\vert }\right]  ,
\end{align*}
where the last equality is valid due to Lemma \ref{Lemma_bipart}.
\end{proof}

Observe that $(-1)^{\left\vert Y\right\vert }x^{\left\vert Z\right\vert
}+(-1)^{\left\vert Z\right\vert }x^{\left\vert Y\right\vert }\neq0$ for any
bipartition $V=Y\cup Z$, which shows together with Lemma \ref{lemma_bipart2}
that
\[
\sum_{F\subseteq E}(-1)^{\left\vert F\right\vert }D\left(  \left(  V,F\right)
,x\right)  \neq0
\]
for any bipartite graph $G=(V,E)$. Moreover, we have the following statement.

\begin{theorem}
\label{theo_main}Let $G=(V,E)$ be a graph. Then%
\[
\sum_{F\subseteq E}(-1)^{\left\vert F\right\vert }D\left(  \left(  V,F\right)
,x\right)  \neq0
\]
if and only if $G$ is bipartite.
\end{theorem}

\begin{proof}
It remains to show that the sum vanishes for non-bipartite graphs. Using
Proposition \ref{prop_delta}, we obtain
\begin{align*}
\sum_{F\subseteq E}(-1)^{\left\vert F\right\vert }D\left(  \left(  V,F\right)
,x\right)   &  =\sum_{F\subseteq E}\sum_{\substack{W\subseteq V \\N_{\left(
V,F\right)  }\left[  W\right]  =V}}(-1)^{\left\vert F\right\vert
}x^{\left\vert W\right\vert }\\
&  =\sum_{W\subseteq V}x^{\left\vert W\right\vert }\sum_{\substack{F\subseteq
E \\N_{\left(  V,F\right)  }\left[  W\right]  =V}}(-1)^{\left\vert
F\right\vert }\\
&  =\sum_{W\subseteq V}x^{\left\vert W\right\vert }\sum_{\substack{F_{1}%
\subseteq\partial(W) \\N_{\left(  V,F_{1}\right)  }\left[  W\right]  =V
}}(-1)^{\left\vert F_{1}\right\vert }\sum_{F_{2}\subseteq E\setminus
\partial(W)}(-1)^{\left\vert F_{2}\right\vert }.
\end{align*}
Since $G$ is not a bipartite graph, the set $F_{2}$ is nonempty, which yields%
\[
\sum_{F_{2}\subseteq E\setminus\partial(W)}(-1)^{\left\vert F_{2}\right\vert
}=0
\]
and hence the statement of the theorem.
\end{proof}

There is also a \textquotedblleft local version\textquotedblright\ for one
direction of Theorem \ref{theo_main}, which can be proved by the same method.

\begin{theorem}
Let $G=(V,E)$ be a graph and $A\subseteq E$ an edge subset such that $\left(
V,A\right)  $ contains an odd cycle. Then%
\[
\sum_{F\subseteq A}(-1)^{\left\vert F\right\vert }D\left(  G-F,x\right)  =0.
\]

\end{theorem}

\subsection{Applications of Spanning Subgraph Expansions}

Let $G=(V,E)$ be a given graph. We define for any edge subset $F$ of $G$,%
\[
h(F)=\sum_{A\subseteq F}(-1)^{\left\vert A\right\vert }D\left(  \left(
V,A\right)  ,x\right)  .
\]
M\"{o}bius inversion yields%
\[
D\left(  \left(  V,F\right)  ,x\right)  =\sum_{A\subseteq F}(-1)^{\left\vert
A\right\vert }h(A).
\]
According to Lemma \ref{Lemma_bipart}, Lemma \ref{lemma_bipart2}, and Theorem
\ref{theo_main}, we define%
\begin{equation}
h(F)=\left\{
\begin{array}
[c]{l}%
x^{l}%
{\displaystyle\prod\limits_{i=1}^{k}}
(-1)^{\left\vert E_{i}\right\vert }\left[  (-1)^{\left\vert Y\cap
V_{i}\right\vert }x^{\left\vert Z\cap V_{i}\right\vert }+(-1)^{\left\vert
Z\cap V_{i}\right\vert }x^{\left\vert Y\cap V_{i}\right\vert }\right]  \text{,
if }\left(  V,F\right)  \text{ is bipartite,}\\
0\text{ otherwise.}%
\end{array}
\right.  \label{h-function}%
\end{equation}
Here the notations are as in Lemma \ref{lemma_bipart2}. We can now conclude
that the domination polynomial of a graph $G=(V,E)$ is a sum of $h$-function
values of spanning bipartite subgraphs, i.e.%
\begin{equation}
D\left(  G,x\right)  =\sum_{\substack{B\subseteq E\\\left(  V,B\right)  \text{
is bipartite}}}h(B). \label{h-sum}%
\end{equation}
The number of dominating sets of $G=(V,E)$ is $D(G,1)$. In order to derive
this number from Equation (\ref{h-sum}), we define $h_{1}$ by substituting
$x=1$ in $h$, that is%
\[
h_{1}(F)=%
{\displaystyle\prod\limits_{i=1}^{k}}
(-1)^{\left\vert E_{i}\right\vert }\left[  (-1)^{\left\vert Y\cap
V_{i}\right\vert }+(-1)^{\left\vert Z\cap V_{i}\right\vert }\right]  .
\]
Observe that $h_{1}(\emptyset)=1$ and $h_{1}(F)\equiv0$ $(\operatorname{mod}%
2)$ for $F\neq\emptyset$, which gives the following statement.

\begin{corollary}
\label{coro_odd}For any graph $G$, the number of dominating sets of $G$ is odd.
\end{corollary}

For alternative proofs of this corollary, see \cite{Brouwer}.

\begin{remark}
In almost the same manner, by substituting $x=-1$ in $h$, we can prove that
$D(G,-1)$ is odd. Moreover, from the Equations (\ref{h-function}) and
(\ref{h-sum}), we obtain%
\[
D(G,-1)=(-1)^{\left\vert V\right\vert }\sum_{\substack{F\subseteq
E\\(V,F)\text{ is bipartite}}}(-1)^{\left\vert F\right\vert }2^{c(F)},
\]
where $c(F)$ denotes here the number of components of $\left(  V,F\right)  $
that have at least one edge.
\end{remark}

\section{Vertex Induced Subgraphs}

Let $G=(V,E)$ be a graph and $W\subseteq V$. We denote by $G\left[  W\right]
$ the \emph{vertex induced subgraph} of $G$:
\[
G\left[  W\right]  =\left(  W,\left\{  \left\{  u,v\right\}  \in E\mid u\in
W\text{ and }v\in W\right\}  \right)  .
\]

\begin{theorem}
\label{theo_vertex_sum}Any connected graph $G=(V,E)$ satisfies%
\[
\sum_{W\subseteq V}(-1)^{\left\vert W\right\vert }D\left(  G\left[  W\right]
,x\right)  =1+\left(  -x\right)  ^{\left\vert V\right\vert }.
\]

\end{theorem}

\begin{proof}
By switching the order of summation, we have%
\begin{align*}
\sum_{W\subseteq V}(-1)^{\left\vert W\right\vert }D\left(  G\left[  W\right]
,x\right)   &  =\sum_{W\subseteq V}(-1)^{\left\vert W\right\vert }%
\sum_{\substack{T\subseteq W\\N_{G\left[  W\right]  }\left[  T\right]
=W}}x^{\left\vert T\right\vert }\\
&  =\sum_{T\subseteq V}x^{\left\vert T\right\vert }\sum_{\substack{W\supseteq
T\\N_{G\left[  W\right]  }\left[  T\right]  =W}}(-1)^{\left\vert W\right\vert
}\\
&  =\sum_{T\subseteq V}x^{\left\vert T\right\vert }\sum_{W:T\subseteq
W\subseteq N_{G\left[  W\right]  }\left[  T\right]  }(-1)^{\left\vert
W\right\vert }\\
&  =\sum_{T\subseteq V}x^{\left\vert T\right\vert }\sum_{W:T\subseteq
W\subseteq N_{G}\left[  T\right]  }(-1)^{\left\vert W\right\vert }\\
&  =\sum_{T\subseteq V}(-x)^{\left\vert T\right\vert }\sum_{Y\subseteq
N_{G}\left[  T\right]  \setminus T}(-1)^{\left\vert Y\right\vert }.
\end{align*}
Since $G$ is connected, the set $N_{G}\left[  T\right]  \setminus T$ is empty
if and only if $T=\emptyset$ or $T=V$. Hence we obtain%
\[
\sum_{T\subseteq V}(-x)^{\left\vert T\right\vert }\sum_{Y\subseteq
N_{G}\left[  T\right]  \setminus T}(-1)^{\left\vert Y\right\vert
}=1+(-x)^{\left\vert V\right\vert }.\text{ }%
\]

\end{proof}

\begin{definition}
Let $G=(V,E)$ be a graph with $n$ vertices. The \emph{type} of $G$ is an
integer partition $\lambda_{G}=\left(  \lambda_{1},...,\lambda_{k}\right)
\vdash n$ that gives the sequence of orders of the components of $G$. We write
$i\in\lambda_{G}$ in order to indicate that $i$ is a part of $\lambda_{G}$.
The number of parts of $\lambda_{G}$ is denoted by $\left\vert \lambda
_{G}\right\vert $.
\end{definition}

Observe that for all $W\subseteq V$ the relation $\left\vert \lambda_{G\left[
W\right]  }\right\vert \leq\alpha(G)$ is satisfied, where $\alpha(G)$ denotes
the independence number of $G$. Theorem \ref{theo_vertex_sum} and Equation
(\ref{product}) immediately imply the following statement.

\begin{corollary}
\label{coro_type}For any graph $G=(V,E)$, we have%
\begin{equation}
\sum_{W\subseteq V}(-1)^{\left\vert W\right\vert }D\left(  G\left[  W\right]
,x\right)  =%
{\displaystyle\prod\limits_{i\in\lambda_{G}}}
\left(  1+(-x)^{i}\right)  . \label{vertex sum}%
\end{equation}

\end{corollary}

The application of the M\"{o}bius inversion to Equation (\ref{vertex sum})
yields%
\begin{align}
D(G,x)  &  =\sum_{W\subseteq V}(-1)^{\left\vert W\right\vert }%
{\displaystyle\prod\limits_{i\in\lambda_{G\left[  W\right]  }}}
\left(  1+(-x)^{i}\right) \nonumber\\
&  =\sum_{W\subseteq V}%
{\displaystyle\prod\limits_{i\in\lambda_{G\left[  W\right]  }}}
\left(  x^{i}+(-1)^{i}\right)  . \label{moeb2}%
\end{align}

\begin{remark}
If we substitute $x=1$ (or $x=-1$) in Equation (\ref{moeb2}) then all the
products are equal to $0$ $(\operatorname{mod}2)$. There is only one
exception, namely the empty product corresponding to $W=\emptyset$, which is
1. This gives an alternative proof of Corollary \ref{coro_odd}.
\end{remark}

We call a graph $G$ \emph{conformal} if all of its components are of even
order. Let $\mathrm{Con}(G)$ be the set of all vertex-induced conformal
subgraphs of $G$ and let $k(G)$ be the number of components of $G$.

\begin{theorem}
\label{theo_con}The number of dominating sets of a graph $G$ satisfies%
\[
d(G)=\sum_{H\in\mathrm{Con}(G)}2^{k(H)}.
\]

\end{theorem}

\begin{proof}
The statement follows from Equation (\ref{moeb2}) by substituting $x=1$. In
this case any component of odd order leads to a zero product, such that only
conformal vertex-induced subgraphs count. Observe that the null graph is
conformal and has no components, which produces the only odd term of the sum,
namely $2^{0}=1$.
\end{proof}

Equation (\ref{moeb2}) offers a possibility to derive a decomposition for the
domination polynomial.

\begin{theorem}
\label{theo_rec}Let $G=(V,E)$ be a graph and $v\in V$. Then%
\[
D(G,x)=D(G-v,x)+\sum_{\substack{\left\{  v\right\}  \subseteq W\subseteq
V\\G\left[  W\right]  \text{ is connected}}}\left(  x^{\left\vert W\right\vert
}+(-1)^{\left\vert W\right\vert }\right)  D\left(  G-N[W],x\right)  .
\]

\end{theorem}

\begin{proof}
We start from Equation (\ref{moeb2}):%
\begin{align*}
D(G,x)  &  =\sum_{W\subseteq V}%
{\displaystyle\prod\limits_{i\in\lambda_{G\left[  W\right]  }}}
\left(  x^{i}+(-1)^{i}\right) \\
&  =\sum_{W\subseteq V\setminus\left\{  v\right\}  }%
{\displaystyle\prod\limits_{i\in\lambda_{G\left[  W\right]  }}}
\left(  x^{i}+(-1)^{i}\right)  +\sum_{\left\{  v\right\}  \subseteq W\subseteq
V}%
{\displaystyle\prod\limits_{i\in\lambda_{G\left[  W\right]  }}}
\left(  x^{i}+(-1)^{i}\right) \\
&  =D(G-v,x)\\
&  +\sum_{\substack{\left\{  v\right\}  \subseteq W\subseteq V\\G\left[
W\right]  \text{ is connected}}}\left(  x^{\left\vert W\right\vert
}+(-1)^{\left\vert W\right\vert }\right)  \sum_{T\subseteq V\setminus N[W]}%
{\displaystyle\prod\limits_{i\in\lambda_{G\left[  T\right]  }}}
\left(  x^{i}+(-1)^{i}\right) \\
&  =D(G-v,x)+\sum_{\substack{\left\{  v\right\}  \subseteq W\subseteq
V\\G\left[  W\right]  \text{ is connected}}}\left(  x^{\left\vert W\right\vert
}+(-1)^{\left\vert W\right\vert }\right)  D\left(  G-N[W],x\right)  .
\end{align*}

\end{proof}

The following statement for the number of dominating sets of $G$ is an
immediate consequence of Theorem \ref{theo_rec}.

\begin{corollary}
Let $G=(V,E)$ be a graph and $v\in V$. Then the difference $d(G)-d(G-v)$ is even.
\end{corollary}

\begin{proof}
When we substitute $x=1$ in Theorem \ref{theo_rec}, then we obtain%
\[
d(G)=d(G-v)+2\sum_{\substack{_{\substack{\left\{  v\right\}  \subseteq
W\subseteq V\\G\left[  W\right]  \text{ is connected}}}\\\left\vert
W\right\vert \text{ is even}}}d\left(  G-N[W]\right)  ,
\]
which gives the desired result.
\end{proof}

\section{Inclusion--Exclusion}

We obtain a different representation of the domination polynomial as a sum
ranging over vertex subsets by counting all vertex subsets of $G=(V,E)$ that
do not dominate the whole vertex set $V$ and applying inclusion-exclusion.

\begin{theorem}
[\cite{DT}]\label{theo_inc_exc}Let $G=(V,E)$ be a graph. Then
\begin{equation}
D(G,x)=\sum_{W\subseteq V}(-1)^{\left\vert W\right\vert }(1+x)^{\left\vert
V\setminus N\left[  W\right]  \right\vert }. \label{incl-excl}%
\end{equation}

\end{theorem}

\begin{corollary}
\label{coro_binomial}The domination polynomial of a graph $G=(V,E)$ with $n$
vertices satisfies%
\[
D(G,x)=\sum_{k=0}^{n}x^{k}\sum_{\substack{W\subseteq V\\\left\vert N\left[
W\right]  \right\vert \leq n-k}}(-1)^{\left\vert W\right\vert }\binom
{n-\left\vert N\left[  W\right]  \right\vert }{k}.
\]

\end{corollary}

\begin{proof}
Using Equation (\ref{incl-excl}), we obtain%
\begin{align*}
D(G,x)  &  =\sum_{W\subseteq V}(-1)^{\left\vert W\right\vert }%
(1+x)^{\left\vert V\setminus N\left[  W\right]  \right\vert }\\
&  =\sum_{W\subseteq V}(-1)^{\left\vert W\right\vert }\sum_{k=0}^{\left\vert
V-N\left[  W\right]  \right\vert }\binom{n-\left\vert N\left[  W\right]
\right\vert }{k}x^{k}\\
&  =\sum_{k=0}^{n}x^{k}\sum_{W\subseteq V}(-1)^{\left\vert W\right\vert
}\binom{n-\left\vert N\left[  W\right]  \right\vert }{k}\\
&  =\sum_{k=0}^{n}x^{k}\sum_{\substack{W\subseteq V\\\left\vert N\left[
W\right]  \right\vert \leq n-k}}(-1)^{\left\vert W\right\vert }\binom
{n-\left\vert N\left[  W\right]  \right\vert }{k}.
\end{align*}

\end{proof}

\begin{remark}
An interesting consequence of Corollary \ref{coro_binomial} is the
characterization of the domination number $\gamma(G)$ of a graph $G=(V,E)$ as
the smallest nonnegative integer $k$ such that the sum%
\[
\sum_{\substack{W\subseteq V\\\left\vert N\left[  W\right]  \right\vert \leq
n-k}}(-1)^{\left\vert W\right\vert }\binom{n-\left\vert N\left[  W\right]
\right\vert }{k}%
\]
does not vanish.
\end{remark}

We call a vertex subset $W\subseteq V$ of a graph $G=(V,E)$ \emph{essential}
in $G$ if $W$ contains the closed neighborhood $N\left[  v\right]  $ of at
least one vertex $v\in V$. We denote by $\mathrm{Ess}(G)$ the family of all
essential sets of $G$.

\begin{theorem}
\label{theo_neighborhood}Let $G=(V,E)$ be a graph with nonempty vertex set.
Then the domination polynomial of $G$ satisfies%
\[
D(G,x)=(-1)^{\left\vert V\right\vert }\sum_{U\in\mathrm{Ess}(G)}%
(-1)^{\left\vert U\right\vert }\left[  (1+x)^{\left\vert \left\{  u\in U\mid
N\left[  u\right]  \subseteq U\right\}  \right\vert }-1\right]  .
\]

\end{theorem}

\begin{proof}
According to Equation (\ref{incl-excl}), we have%
\begin{align*}
D(G,x)  &  =\sum_{W\subseteq V}(-1)^{\left\vert W\right\vert }%
(1+x)^{\left\vert V\setminus N\left[  W\right]  \right\vert }\\
&  =\sum_{U\subseteq V}(-1)^{\left\vert V\right\vert -\left\vert U\right\vert
}(1+x)^{\left\vert V\setminus N\left[  V\setminus U\right]  \right\vert }\\
&  =\sum_{U\subseteq V}(-1)^{\left\vert V\right\vert -\left\vert U\right\vert
}(1+x)^{\left\vert \left\{  u\in U\mid N\left[  u\right]  \subseteq U\right\}
\right\vert }.
\end{align*}
In order to see the last equality, we verify%
\begin{align*}
N\left[  V\setminus U\right]   &  =\bigcup\limits_{v\in V\setminus U}N\left[
v\right] \\
&  =(V\setminus U)\cup\left\{  u\in U\mid N\left[  u\right]  \cap(V\setminus
U)\neq\emptyset\right\} \\
&  =(V\setminus U)\cup\left\{  u\in U\mid N\left[  u\right]  \nsubseteq
U\right\}
\end{align*}
and consequently,%
\begin{align*}
V\setminus N\left[  V\setminus U\right]   &  =V\setminus\left[  (V\setminus
U)\cup\left\{  u\in U\mid N\left[  u\right]  \nsubseteq U\right\}  \right] \\
&  =U\setminus\left\{  u\in U\mid N\left[  u\right]  \nsubseteq U\right\} \\
&  =\left\{  u\in U\mid N\left[  u\right]  \subseteq U\right\}  .
\end{align*}

All polynomials of the form $(1+x)^{\left\vert \left\{  u\in U\mid N\left[
u\right]  \subseteq U\right\}  \right\vert }$ have the constant term 1. As
$V\neq\emptyset$, the constant term in%
\[
\sum_{U\subseteq V}(-1)^{\left\vert V\right\vert -\left\vert U\right\vert
}(1+x)^{\left\vert \left\{  u\in U\mid N\left[  u\right]  \subseteq U\right\}
\right\vert }%
\]
vanishes, which gives%
\[
D(G,x)=\sum_{U\subseteq V}(-1)^{\left\vert V\right\vert -\left\vert
U\right\vert }\left[  (1+x)^{\left\vert \left\{  u\in U\mid N\left[  u\right]
\subseteq U\right\}  \right\vert }-1\right]  .
\]
If $U$ is a non-essential set of $G$ then we have $\left\{  u\in U\mid
N\left[  u\right]  \subseteq U\right\}  =\emptyset$ and hence
$(1+x)^{\left\vert \left\{  u\in U\mid N\left[  u\right]  \subseteq U\right\}
\right\vert }=1$. Consequently, all non-vanishing terms correspond to
essential sets, yielding the statement of the theorem.
\end{proof}

Another interesting consequence of Theorem \ref{theo_inc_exc} is the following
relation between $D(G,x)$ and $D\left(  G,\frac{1}{x}\right)  $.

\begin{theorem}
\label{theo_reciprocal}Let $G=(V,E)$ be a graph. Then%
\[
D(G,x)=(1+x)^{\left\vert V\right\vert }\sum_{W\subseteq V}\left(  \frac
{-x}{1+x}\right)  ^{\left\vert W\right\vert }D\left(  G\left[  W\right]
,\frac{1}{x}\right)  .
\]

\end{theorem}

\begin{proof}
We consider the right-hand side of the equation from the theorem. Substituting
$D\left(  G\left[  W\right]  ,\frac{1}{x}\right)  $ according to the
definition of the domination polynomial yields%
\begin{align*}
&  (1+x)^{\left\vert V\right\vert }\sum_{W\subseteq V}\left(  \frac{-x}%
{1+x}\right)  ^{\left\vert W\right\vert }\sum_{\substack{T\subseteq
W\\N_{G\left[  W\right]  }\left[  T\right]  =W}}x^{-\left\vert T\right\vert
}\\
&  =(1+x)^{\left\vert V\right\vert }\sum_{W\subseteq V}\left(  \frac{-x}%
{1+x}\right)  ^{\left\vert W\right\vert }\sum_{T:T\subseteq W\subseteq
N_{G}\left[  T\right]  }x^{-\left\vert T\right\vert }.
\end{align*}
Switching the order of summation gives%
\[
(1+x)^{\left\vert V\right\vert }\sum_{T\subseteq V}x^{-\left\vert T\right\vert
}\sum_{W:T\subseteq W\subseteq N_{G}\left[  T\right]  }\left(  \frac{-x}%
{1+x}\right)  ^{\left\vert W\right\vert }.
\]
Now we define $U=W\setminus T$ and substitute $W=U\cup T$, yielding%
\begin{align*}
&  (1+x)^{\left\vert V\right\vert }\sum_{T\subseteq V}x^{-\left\vert
T\right\vert }\sum_{U\subseteq N_{G}\left[  T\right]  \setminus T}\left(
\frac{-x}{1+x}\right)  ^{\left\vert U\right\vert +\left\vert T\right\vert }\\
&  =(1+x)^{\left\vert V\right\vert }\sum_{T\subseteq V}\left(  \frac{-1}%
{1+x}\right)  ^{\left\vert T\right\vert }\sum_{U\subseteq N_{G}\left[
T\right]  \setminus T}\left(  \frac{-x}{1+x}\right)  ^{\left\vert U\right\vert
},
\end{align*}
which simplifies via the binomial theorem to%
\begin{align*}
&  (1+x)^{\left\vert V\right\vert }\sum_{T\subseteq V}\left(  \frac{-1}%
{1+x}\right)  ^{\left\vert T\right\vert }\left(  1-\frac{x}{1+x}\right)
^{\left\vert N_{G}\left[  T\right]  \right\vert -\left\vert T\right\vert }\\
&  =(1+x)^{\left\vert V\right\vert }\sum_{T\subseteq V}\left(  -1\right)
^{\left\vert T\right\vert }\left(  1+x\right)  ^{-\left\vert N_{G}\left[
T\right]  \right\vert }.
\end{align*}
The statement follows now by Theorem \ref{theo_inc_exc}.
\end{proof}

The following statement can be shown by substituting $x=1$ in Theorem
\ref{theo_reciprocal}.

\begin{corollary}
Let $G=(V,E)$ be a graph. The numbers of dominating sets of vertex-induced
proper subgraphs of $G$ satisfy%
\[
\sum_{W\subset V}(-1)^{\left\vert W\right\vert }\frac{d(G[W])}{2^{\left\vert
W\right\vert }}=0.
\]

\end{corollary}

\section{Conclusions and Open Problems}

The domination polynomial of a graph can be expressed as a sum of quite simple
polynomials of vertex-induced or spanning subgraphs. In case of spanning
subgraphs, we can show that the domination polynomial depends only on
bipartite spanning subgraphs.

There remain interesting open problems for further research in this field. The
first one concerns the number of dominating sets of a graph given by Theorem
\ref{theo_con}.

\begin{problem}
The simple formula%
\[
d(G)=\sum_{H\in\mathrm{Con}(G)}2^{k(H)}%
\]
suggests that there is a bijection between subsets of components of conformal
graphs and dominating sets of $G$. Is there a bijective proof for Theorem
\ref{theo_con}? What is the best way to enumerate the set $\mathrm{Con}(G)$?
\end{problem}

In Corollary \ref{coro_type}, we showed that the type of a subgraph yields the
essential information for a representation of $D(G,x)$ as a sum over
vertex-induced subgraphs. Here it seems interesting to investigate whether we
need all vertex-induced subgraphs in order to derive the domination polynomial.

\begin{problem}
Components of $G\left[  W\right]  $ that have odd order lead to cancellation
of terms of the sum in Equation (\ref{moeb2}),%
\[
D(G,x)=\sum_{W\subseteq V}%
{\displaystyle\prod\limits_{i\in\lambda_{G\left[  W\right]  }}}
\left(  x^{i}+(-1)^{i}\right)  .
\]
Is there a way to identify those cancelling terms?
\end{problem}

\begin{problem}
In Theorem \ref{theo_neighborhood}, we showed that the restriction to
essential sets is sufficient in order to compute the domination polynomial of
a graph. Can we reduce the number of terms needed to derive $D(G,x)$ further?
\end{problem}

Further topics of interest for future research include the investigation of
special graph classes with respect to the given representations of the
domination polynomial and the application of these representations to special
graph classes. Since bipartite graphs play an important role for the
representation of the domination polynomial, we conjecture that also matchings
have a close relation to dominating sets. However, until now all attempts to
find a sum representation of $D(G,x)$ based on matchings of $G$ failed.


\begin{thebibliography}{9}                                                                                                %


\bibitem {AAOP}Saieed Akbari, Saeid Alikhani, Mohammad Reza Oboudi and
Yee-Hock Peng: \emph{On the zeros of domination polynomial of a graph}, in
Brualdi, Richard A. (ed.) et al., Combinatorics and graphs. Selected papers
based on the presentations at the 20th anniversary conference of IPM on
combinatorics, Tehran, Iran, May 15--21, 2009, Contemporary Mathematics 531,
109-115 (2010).

\bibitem {AP}Saeid Alikhani and Yee-Hock Peng: \emph{Dominating sets and
domination polynomials of paths}, International Journal of Mathematics and
Mathematical Sciences, Hindawi, Volume 2009 (2009), Article ID 542040, 10
pages, doi:10.1155/2009/542040.

\bibitem {AP2}Saeid Alikhani and Yee-Hock Peng: \emph{Dominating sets and
domination polynomials of certain graphs II}, Opuscula Mathematica, Vol. 30
(2010), No. 1, 37-51.

\bibitem {AAP}Saieed Akbari, Saeid Alikhani, Yee-hock Peng:
\emph{Characterization of graphs using domination polynomials}, European
Journal of Combinatorics, 31 (2010) 7, 1714-1724.

\bibitem {Arocha}Jorge Luis Arocha and Bernardo Llano: \emph{Meanvalue for the
matching and dominating polynomial}, Discuss. Math. Graph Theory, 20 (2000),
pp. 57--69.

\bibitem {Brouwer}Andries E. Brouwer: \emph{The number of dominating sets of a
finite graph is odd}, preprint, http://www.win.tue.nl/\symbol{126}aeb/preprints/domin2.pdf.

\bibitem {DT}Klaus Dohmen and Peter Tittmann: \emph{Domination reliability},
The Electronic Journal of Combinatorics 19 (2012), \#P15.


\end{thebibliography}
\end{document}